\documentclass[11pt]{amsart}
\usepackage{amssymb,amsmath}
\usepackage[curve,matrix,arrow]{xy}
\usepackage[normalem]{ulem}
\usepackage{xcolor}
\newtheorem{thm}{Theorem}[section]
\newtheorem{lemma}[thm]{Lemma}
\newtheorem{prop}[thm]{Proposition}
\newtheorem{cor}[thm]{Corollary}

\theoremstyle{definition}
\newtheorem{defn}[thm]{Definition}

\newtheorem{remark}[thm]{Remark}
\newtheorem{example}[thm]{Example}

\newcommand{\hf}[1]{\operatorname{\mathbf H}_{#1} \! \mathfrak F}

\def\cd{\operatorname{cd}\nolimits}
\def\vcd{\operatorname{vcd}\nolimits}

\newcommand{\ul}[1]{\underline{#1}}
\newcommand\wh[1]{\widehat{#1}}
\newcommand{\ovl}[1]{\overline{#1}}

\newcommand{\qbox}[1]{\quad\hbox{#1}\quad}
\def\dst{\displaystyle}

\def\im{\operatorname{im}}

\def\id{\operatorname{Id}}
\def\Res{\operatorname{Res}\nolimits}
\def\Ind{\operatorname{Ind}\nolimits}

\def\Inf{\operatorname{Inf}\nolimits}

\def\Coind{\operatorname{Coind}\nolimits}
\renewcommand\mod{\operatorname{mod}}
\def\Mod{\operatorname{Mod}\nolimits}

\def\PHom{\operatorname{PHom}}
\def\Hom{\operatorname{Hom}\nolimits}

\def\proj{\operatorname{(proj)}\nolimits}
\def\Ext{\operatorname{Ext}\nolimits}
\def\End{\operatorname{End}\nolimits}
\def\Aut{\operatorname{Aut}\nolimits}
\def\coker{\operatorname{coker}\nolimits}

\def\SL{\operatorname{SL}\nolimits}

\newcommand{\res}[2]{\!\!\downarrow^{#1}_{#2}~}
\newcommand{\ind}[2]{\!\!\uparrow_{#1}^{#2}~}

\DeclareMathOperator{\tac}{tac}
\DeclareMathOperator{\inc}{Inc}
\DeclareMathOperator{\GPmp}{\underline{GP}}
\DeclareMathOperator{\compres}{CompRes}
\DeclareMathOperator{\Proj}{Proj}
\DeclareMathOperator{\Stab}{Stab}
\DeclareMathOperator{\spli}{spli}
\DeclareMathOperator{\silp}{silp}
\DeclareMathOperator{\projdim}{projdim}
\DeclareMathOperator{\injdim}{injdim}
\DeclareMathOperator{\findim}{findim}
\DeclareMathOperator{\gldim}{gldim}
\DeclareMathOperator{\modproj}{\underline{Mod}}
\DeclareMathOperator{\GP}{GP}
\DeclareMathOperator{\smod}{\widehat{\Mod}}
\DeclareMathOperator{\shom}{\widehat{\Hom}}
\DeclareMathOperator{\send}{\widehat{\End}}
\DeclareMathOperator{\saut}{\widehat{\Aut}}
\DeclareMathOperator{\ev}{ev}
\DeclareMathOperator{\rest}{res}
\DeclareMathOperator{\tor}{tor}

\def\Z{\mathbb Z}

\title{The Stable Category and Invertible Modules for Infinite Groups}
\date{\today}
\author{Nadia Mazza}
\address{Department of Mathematics and Statistics\\
	Lancaster University\\
	Lancaster LA1 4YF\\
United Kingdom}
\email{n.mazza@lancaster.ac.uk}
\author{Peter Symonds}
\thanks{Second author partially supported by an International Academic Fellowship from the Leverhulme Trust.}
\address{Department of Mathematics\\
	University of Manchester\\
	Manchester M13 9PL\\
	United Kingdom}
\email{Peter.Symonds@manchester.ac.uk}

\date{\today}

\begin{document}
	
	\begin{abstract} We construct a well-behaved stable category of modules for a large class of infinite groups. We then consider its Picard group, which is the group of invertible (or endotrivial)  modules. We show how this group can be calculated when the group  acts on a tree with finite stabilisers.
		\end{abstract}
	
\maketitle

\section{Introduction}


In the modular representation theory of finite groups, the stable module category is of fundamental importance. It is normally constructed by quotienting out all the morphisms that factor through a projective module, but it can also be characterised as the largest quotient category on which the syzygy operator is well defined and invertible. It is a triangulated category when the distinguished triangles are taken to be all the triangles that are isomorphic to the image of a short exact sequence of modules.

In the case of infinite groups we will construct a stable category
that has these properties, at least for a fairly large class of groups that
we call the groups of type $\Phi_k$ (see \cite{talelli}). This class includes all groups of finite virtual cohomological dimension over $k$ and all groups that admit a
finite dimensional classifying space for proper actions, $\underbar{E}G$, see Corollary~\ref{cor:g-cw-cx}. However, just quotienting out the projectives is usually not sufficient. We also describe various other possible constructions of this category and show that they all give equivalent results.

Our methods are similar to those of Ikenaga \cite{ikenaga}, even though he does not mention a stable category and is instead interested in generalising the Tate-Farrell cohomology theory to a larger class of groups than those of finite virtual cohomological dimension. Benson \cite{benson} has also described a stable category for infinite groups in the same spirit as ours, but he places restrictions on the modules rather than on the groups. The work of Cornick, Kropholler and coauthors is closely related too, but it is more homological in flavour; see \cite{cekt} for a recent account and more details of the history of the subject.

This stable category has a well-behaved tensor product over the ground ring $k$, so is a tensor triangulated category in the language of Balmer. One important invariant of such a category is its Picard group. In this case the elements are the stable isomorphism classes of $kG$-modules $M$ for which there exists a module $N$ such that $M \otimes_kN$ is in the stable class of the trivial module $k$. These form a group under tensor product, which we denote by $T_k(G)$. For finite groups this is known as the group of endotrivial modules and has been extensively studied by many authors (see \cite{thevenaz} for an excellent survey). Such modules have been classified for $p$-groups in characteristic $p$ and also for many families of general finite groups \cite{CHM,CMN,CMN2,CT}.

We develop the basic theory of these modules in the infinite case and provide some tools for calculating the Picard group  $T_k(G)$. In particular, we have a formula whenever $G$ acts on a tree with finite stabilisers.

One original justification for studying these modules in the finite case, given by Dade when he introduced them \cite{dade}, was that they form a class of modules that is ``small enough to be classified and large enough to be useful''. We illustrate how this remains true in the infinite case at the end of the paper, by calculating some examples.

Some notes for a summer school course based on this material will appear in \cite{vanc}.

\section{Groups of Type $\Phi_k$}\label{sec:phi}

In this paper, $k$ will always be a commutative noetherian ring of finite global
dimension; the fundamental case is when $k$ is a field of positive characteristic. All modules will be $k$-modules; sometimes a fixed $k$ will be understood and we will omit it from the notation. We will consider groups $G$, usually infinite, and the category of all $kG$-modules $\Mod (kG)$.

Recall that the projective dimension of a $kG$-module, $\projdim_{kG}M$, is the shortest possible length of a projective resolution of $M$ ($\infty$ if there is no resolution of finite length). The global dimension of $k$, $\gldim k$, is the supremum of the projective dimensions of all $k$-modules.

\begin{defn}
	A group $G$ is of type $\Phi_k$ if it has the property that a $kG$-module is of finite projective dimension if and only if its restriction to any finite subgroup is of finite projective dimension.
\end{defn}

\begin{defn}
	The finitistic dimension of $kG$ is $$\findim kG = \sup \{ \projdim_{kG}M | \projdim_{kG}M < \infty \}.$$
	\end{defn}

\begin{lemma}
	\label{le:fd}
	If the group $G$ is finite and the $kG$-module $M$ has finite projective dimension then $\projdim_{kG} M \leq \gldim k$.
	\end{lemma}

\begin{proof}
	Let $d = \gldim k$ and consider the $d$th syzygy $\Omega ^dM$ of $M$. This is projective over $k$, so its projective resolution is split over $k$. The modules in the resolution are summands of modules induced from the trivial subgroup  and induction is equivalent to coinduction for finite groups, so these modules are injective relative to the trivial subgroup. The projective resolution can be taken to be of finite length so, by relative injectivity, it can be split starting from the left, thus $\Omega^dM$ must be projective.
	\end{proof}

\begin{lemma}
	\label{lem:findim}
	If $G$ is of type $\Phi_k$ then $\findim kG$ is finite.
	\end{lemma}

\begin{proof}
	If $\findim kG$ is not finite, then for each positive integer $i$ we can find a $kG$-module $M_i$ such that $i \leq \projdim_{kG}M_i < \infty$. Over any finite subgroup $F$ we have $ \projdim_{kF}M_i < \infty$, so by Lemma~\ref{le:fd} we obtain $ \projdim_{kF}M_i \leq \gldim k$. Let $M = \oplus_iM_i$; then for any finite subgroup $F$ we have $ \projdim_{kF}M \leq \gldim k$, yet $ \projdim_{kG}M = \infty$, so $G$ cannot be of type $\Phi_k$.
	\end{proof}

 The class of groups of type $\Phi_{\mathbb Z}$ was introduced by Talelli \cite{talelli}. The class of groups of type $\Phi_k$ is closed under subgroups and contains many naturally occurring groups; in particular, we have the following result.

\begin{prop}\label{prop:g-cw-cx} 
	Suppose that there exists a number $m$ and an exact complex of $kG$-modules $$0 \rightarrow C_n \rightarrow \cdots \rightarrow C_0 \rightarrow k \rightarrow 0,$$
	such that each $C_i$ a summand of a sum of modules of the form $k \! \uparrow_H^G$ with $H$ of type $\Phi_k$ and $\findim kH \leq m$. Then $G$ is of type
	$\Phi_k$ and $\findim kG \leq n+m$.  
\end{prop}

\begin{proof}
	Let $M$ be a $kG$-module that is of finite projective dimension on restriction to any finite subgroup of $G$. The exact complex in the statement of the proposition is split over $k$, so it remains exact on tensoring with $M$.
	
	Each $C_i \otimes_k M$ is a summand of a sum of terms of the form $(k \! \uparrow_H^G) \otimes_k M \cong M \! \downarrow_H \! \uparrow ^G$. By the hypotheses on $H$, we know that $M \! \downarrow _H$ is of finite projective dimension and $\projdim_{kH}M \leq m$, hence $\projdim_{kG} M \! \downarrow_H \! \uparrow ^G \leq m$.
	
	Setting $D_i=C_i \otimes_k M$ we obtain modules $D_i$ with $\projdim_{kG}D_i\leq m$ and an exact complex $$0 \rightarrow D_n \rightarrow \cdots \rightarrow D_0 \rightarrow M \rightarrow 0.$$ An easy induction on $n$ now shows that $\projdim_{kG}M \leq n+m$.
\end{proof}

Recall that a group $G$ is of finite cohomological dimension over $k$ (finite $\cd_k$) if $\projdim_{kG}k < \infty$ and is of finite virtual cohomological dimension over $k$ (finite $\vcd_k$) if it has a subgroup of finite index that is of finite cohomological dimension over $k$ (see \cite[VII.11]{brown}).

\begin{cor}
	\label{cor:g-cw-cx}
	A group $G$ is of type $\Phi_k$ if either of the following conditions hold:
	\begin{enumerate}
		\item
		there is a finite dimensional contractible $G$-CW-complex with finite stabilisers, or
		\item
		$G$ is of finite $\vcd_k$.
		\end{enumerate}
	\end{cor}

\begin{proof}
	In case (1) we just use the associated $G$-CW-chain complex in Proposition~\ref{prop:g-cw-cx}. In case (2) the complex is constructed by the method of Swan, following Serre \cite[9.2]{swan}.
	\end{proof}

Note that a group $G$ that admits a
finite dimensional classifying space for proper actions, $\underbar{E}G$, satisfies condition (1) above. It has been conjectured by Talelli that $G$ having a finite dimensional $\underbar{E}G$  is actually equivalent to $G$ being of type $\Phi_{\mathbb Z}$ \cite[Conj.\ A]{talelli}.
	
\begin{example} A free abelian group of infinite rank is not of type $\Phi_k$ for any $k$, so neither is any group that contains it.
	\end{example}

\section{The Stable Category}

Write
$\Mod(kG)$ for the category of all $kG$-modules, possibly infinitely generated, and 
$\ul{\Mod}(kG)$ the quotient category with the same objects but morphisms
$$\ul{\Hom}_{kG}(M,N)=\Hom_{kG}(M,N)/\PHom_{kG}(M,N),$$ where 
$\PHom_{kG}(M,N)$ is the submodule of $kG$-homomorphisms $M\to N$ which
factor through a projective module. 

The next observation is well known; for a proof see e.g.\  \cite[\S2]{HP}.

\begin{lemma}\label{lem:stab-iso}
	Suppose $M\cong N$ in $\ul{\Mod}(kG)$. Then there exist projective
	$kG$-modules $P,Q$ such that $M\oplus P\cong N\oplus Q$ in $\Mod(kG)$. 
	
	If $M$ is finitely generated then $Q$ can be taken to be finitely generated and if $N$ is finitely generated then $P$ can also be taken to be finitely generated.
\end{lemma}

When $G$ is finite and $k$ is a field then $\ul{\Mod}(kG)$ is taken to be the stable category, but this does not work well in the infinite case, because $\Omega$ might not be invertible. We shall introduce various possible definitions of a stable category for infinite groups and show that they all agree for groups of type $\Phi_k$.

Recall that there is a natural map $\Omega: \ul{\Hom}_{kG}(M,N) \rightarrow \ul{\Hom}_{kG}(\Omega M, \Omega N)$.

\begin{defn}[\cite{BC}]
	Let $\Stab(kG)$ be the category with all $kG$-modules as objects and morphisms given by $\Hom_{\Stab(kG)}(M,N) = \lim_{n \to \infty} \Hom_{\modproj(kG)}(\Omega^nM,\Omega^nN)$.
	\end{defn}

Note that a $kG$-module has image 0 in $\Stab(kG)$ if and only if it is of finite projective dimension.

This definition can be difficult to work with. In particular, a map in $\Hom_{\Stab(kG)}(M,N)$ might not correspond to any map in $\Hom_{\Mod(kG)}(M,N)$. However, $\Stab(kG)$ is clearly the largest quotient of $\Mod(kG)$ on which $\Omega$ is well defined and invertible. It is quite possible that $\Stab(kG)=0$, for example if $G$ has finite $\cd_k$. 

\begin{defn}
	An acyclic complex of projectives is an unbounded chain complex of projective modules $P_*$ that is exact everywhere. It is totally acyclic if, in addition, $\Hom_{kG}(P_*,Q)$ is acyclic for any projective module $Q$. The category of totally acyclic complexes of projective $kG$-modules and chain maps up to homotopy will be denoted by $K_{\tac}(kG)$.
	\end{defn}

The category $K_{\tac}(kG)$ is naturally triangulated in the usual way for categories based on chain complexes. The triangle shift is given by $[-1]$, where $[n]$ denotes degree shift by $n$.

\begin{defn}
	A Gorenstein projective module is a module that is isomorphic to a kernel in a totally acyclic complex of projectives. The full subcategory of $\Mod (kG)$ on the Gorenstein projective modules will be denoted by $\GP(kG)$ and the full subcategory in $\modproj (kG)$ by $\GPmp(kG)$.
	\end{defn}

In some cases Gorenstein projective modules are easy to recognise.

\begin{lemma}
Suppose that $G$ is of finite $\vcd_k$. The following conditions on a $kG$-module $M$ are equivalent.
\begin{enumerate}
	\item
	$M$ is Gorenstein projective,
	\item
	the restriction of $M$ to some subgroup of finite index that is of finite $\cd_k$ is projective,
	\item
	the restriction of $M$ to any subgroup that is of finite $\cd_k$ is projective.
	\end{enumerate}
	\end{lemma}

\begin{proof} Suppose that (1) holds and that $H$ is a subgroup of finite $\cd_k$. Let $n = \cd_k(H)$ and $d= \gldim k$; there is a $kG$-module $N$ such that $M \cong \Omega^{n+d}N \cong (\Omega^nk) \otimes_k (\Omega^dN)$ modulo projectives. Now $\Omega^dN$ is projective over $k$ and, on restriction to $H$, $\Omega^nk$ is projective, thus $M$ is projective over $kH$ and (3) holds. Trivially, (3) implies (2). If (2) holds then the construction in \cite[X.2.1]{brown} produces a complete resolution with $M$ as a kernel, so (1) holds. 
	\end{proof}

In particular, if $G$ is finite then a module is Gorenstein projective if and only if it is projective over $k$. If $k$ is also a field then all $kG$-modules are Gorenstein projective.

\begin{remark} The Gorenstein projective modules correspond to the
  modules called cofibrant in Benson's treatment \cite{benson}, at
  least for groups of type $\Phi_k$ (see \cite{bdt,talelli}).
	\end{remark}

The category $\GP(kG)$ is a Frobenius category, hence $\GPmp(kG)$ is naturally triangulated \cite[2.2]{DEH}. The shift is $\Omega^{-1}$, which is obtained by taking the kernel in degree one less in the totally acyclic complex used to show that the module is Gorenstein projective. 

\begin{defn}\label{def:fctor}
There is a functor $\Omega^0: K_{\tac}(kG) \rightarrow \GPmp(kG)$ obtained by taking the kernel of the boundary map in degree 0.  There is also a natural inclusion functor $\inc : \GPmp(kG) \rightarrow \modproj (kG)$.  
\end{defn}

\begin{defn}
	\label{def:g-coho}
 A {\em complete resolution} of a $kG$-module $M$ consists of a totally acyclic chain complex $F_*$ of
projective $kG$-modules and a projective resolution $P_*$ of $M$, together with a map $F_* \rightarrow P_*$ and an integer $n$ such that the map is an isomorphism in all degrees $n$ and above.
$$
\xymatrix{
&\dots\ar[r]&F_{n+1}\ar[r]\ar@{=}[d]&F_n\ar[r]\ar@{=}[d]&F_{n-1}\ar[r]\ar[d]&\dots\ar[r]&F_0\ar[r]\ar[d]&F_{-1}\ar[r]&\dots\\
&\dots\ar[r]&P_{n+1}\ar[r]&P_n\ar[r]&P_{n-1}\ar[r]&\dots\ar[r]&P_0\ar[r]^\varepsilon&M\ar[r]&0}
$$  
The index $n$ is called the {\em coincidence index}. 
\end{defn}

In the literature this is sometimes called a complete resolution in the strong sense, because of the word totally. We will often abuse the terminology by referring to just  $F_*$ as the complete resolution.

It is known that two complete resolutions of the same module must be chain homotopy equivalent (i.e.\ the $F_*$ are chain homotopy equivalent). A map between two $P_*$ induces a map of the corresponding $F_*$, unique up to homotopy (cf.\ \cite[Lemma~3, Proposition~13]{ikenaga}).

In the same way, we can define a complete resolution of any chain
complex $M_*$ with only finitely many non-zero homology groups: we
just take $P_*$ to be a projective resolution of $M_*$, i.e. a
quasi-isomorphism $P_*\to M_*$.

In general, a module might not possess a complete resolution, but for groups of type $\Phi_k$ they always exist.

\begin{thm}
	\label{thm:compres}
	If $G$ is a group of type $\Phi_k$ then any complex of $kG$-modules with only finitely many non-zero homology groups has a complete resolution.
	\end{thm}

We defer the proof to the end of this section.

Thus, for groups of type $\Phi_k$, we obtain a functor $\compres: D^b(\Mod(kG)) \rightarrow K_{\tac}(kG)$, where $D^b(\Mod(kG))$ is the derived category of complexes of $kG$-modules with only finitely many non-zero homology groups. It is easy to see that $K^b(\Proj(kG))$ is in the kernel, where $K^b(\Proj(kG))$ is the homotopy category of bounded complexes of projective $kG$-modules, so we have a functor on the Verdier quotient $\compres: D^b(\Mod(kG))/K^b(\Proj(kG)) \rightarrow K_{\tac}(kG)$.

There is a  functor $[0]:\Stab(kG) \rightarrow
D^b(\Mod(kG))/K^b(\Proj(kG))$. On modules it can be given by regarding
a module as the degree 0 term of a complex that is 0 elsewhere. For
morphisms it is better to regard this slightly differently. Take a
projective resolution $P_M$ of the module $M$ and consider the
truncation $\sigma_{\geq n}(P_M)$ for some $n \geq 0$ (see \cite[1.2.7]{weibel}); it is clearly a projective resolution of $\Omega^nM[-n]$ and it is also equal to $P_M$ in the Verdier localisation. Thus we have $M \cong \Omega^nM[-n]$ in $D^b(\Mod(kG))/K^b(\Proj(kG))$. A morphism $f \in \Hom _{\Stab(kG)}(M,N)$ must correspond to a homomorphism $f' \in \Hom _{\Stab(kG)}(\Omega^nM,\Omega^nN)$ for some $n$, so $f[0]$ can be defined as $f'[-n]: \Omega^nM[-n] \rightarrow \Omega^nN[-n]$.

The next theorem is essentially due to Buchweitz \cite{Buc}, at least for parts (2--4); there are many variants in the literature.

\begin{thm}
	\label{thm:catequiv}
	For a group of type $\Phi_k$, the following categories are equivalent:
	\begin{enumerate}
		\item
		$\Stab(kG)$,
		\item
		$\GPmp(kG)$,
		\item
		$K_{\tac}(kG)$,
		\item
		$D^b(\Mod(kG))/K^b(\Proj(kG))$.
		\end{enumerate}
	The equivalences are obtained from the functors $\inc$,
        $\Omega^0$, $\compres$ and $[0]$ introduced in Definition~\ref{def:fctor}. These are equivalences of triangulated categories, except for those involving $\Stab(kG)$, where a triangulated structure has not yet been defined.
	\end{thm}

\begin{proof}
	The equivalences not involving $\Stab(kG)$ are shown in
        \cite[4.16]{beligiannis} (replacing $\inc$ and $[0]$ by
        $[0]\inc$).
        The theorem will follow when we show that the cyclic permutations of the
        composition $\inc\Omega^0\compres[0]$ are naturally isomorphic
        to the identity. But the cases with $[0]$ and $\inc$ adjacent follow from \cite[4.16]{beligiannis} , so we only have
        to check $\inc\Omega^0\compres[0]$ itself. By definition, this sends $M$ to the
        degree 0 kernel $\tilde{M}$ in its complete resolution. But
        large enough syzygies of $M$ and $\tilde{M}$ are identical in
        $\Stab(kG)$.  
	\end{proof}

Of course, we can use these equivalences to define a triangulated
structure on $\Stab(kG)$. The distinguished triangles are all the
triangles isomorphic to a short exact sequence of modules. The shift
$\Omega^{-1}M$ is obtained by finding a complete resolution of the
module $M$ and then taking the kernel in degree $-1$. The kernel in
degree 0 we denote by $\tilde{M}$; from the definition of a complete
resolution, $\tilde{M}$ is Gorenstein projective and it comes with a
map of modules $\alpha_M \! : \tilde{M} \rightarrow M$ that is an
isomorphism in $\Stab(kG)$.

\begin{defn}\label{def:g-approx}
We call the map $\alpha_M:\tilde M\to M$ the Gorenstein approximation
of $M$.
\end{defn}

We can consider any of these categories to be the stable module
category of $kG$ and we will use the notation $\smod(kG)$ when we do
not wish to specify which one. We will use the symbol $\simeq$ to
denote isomorphism in $\smod(kG)$. 

A $kG$-module is called a lattice if it is projective over $k$. The
full subcategory of $\Stab(kG)$ on the lattices is also equivalent to
$\smod(kG)$, because $\Stab(kG)$ also contains the full subcategory
$\GPmp(kG)$ (and so
$\GPmp(kG)\subseteq\Stab(kG)\subseteq\smod(kG)$ where $\GPmp(kG)$ and
$\smod(kG)$ are equivalent categories). 

\begin{lemma}\label{lem:nat-onto}
	If a $kG$-module $M$ is Gorenstein projective then for any $kG$-module $N$ the natural map $\Hom_{kG}(M,N) \rightarrow \Hom_{\Stab(kG)}(M,N)$ is surjective.
	\end{lemma}

\begin{proof}
	Any $f \in \Hom_{\Stab(kG)}(M,N)$ is the image under $\inc$ of some $\tilde{f} \in \Hom_{\GPmp(kG)}(\tilde{M} , \tilde{N})$, which we can compose with the natural map $\alpha_N \! : \tilde{N} \rightarrow N$. Since $M$ itself is Gorenstein projective, we have $\tilde{M}  \cong  M$ in $\GPmp(kG)$, so we obtain a map $M \rightarrow N$ in $\modproj(kG)$, which must be the image of one in $\Mod(kG)$.
	\end{proof}

Recall that for $H \leq G$ there are two functors $\Ind_H^G$ and $\Coind_H^G$ from $\Mod(kH)$ to $\Mod(kG)$, the left and right adjoints of restriction. They satisfy the identity
\[
\Hom_{k}( M, \Coind_H^G N) \cong \Coind_H^G \Hom_{k}(M,N) \cong \Hom_{k}(\Ind^G_H,M), \quad M,N \in \Mod(kG),
\]
where $\Hom_k(M,N)$ is considered as a $kG$-module in the usual way and all restrictions are implicit.

\begin{lemma}
	\label{lem:coind}
\begin{enumerate}
	\item
	A $kG$-module is projective if and only if it is a summand of a module induced from a free module for the trivial subgroup; it is injective if and only if it is a summand of a module coinduced from an injective module for the trivial subgroup.
	\item
	A $kG$ module induced from the trivial subgroup has finite projective dimension. If $G$ has type $\Phi_k$ then a module coinduced from the trivial subgroup has finite projective dimension.
	\end{enumerate}
\end{lemma}

\begin{proof}
	The first part is trivial. For the second part, if $M = \Ind^G_1X$, take a projective resolution of $X$ over $k$. This is of finite length since $k$ has finite global dimension. Because $\Ind^G_1$ is exact we can apply it to obtain a projective resolution of finite length for $M$.
	
	If $N = \Coind^G_1Y$, note that, since $G$ is assumed to be of type $\Phi_k$, we only have to check finite projective dimension on restriction to finite subgroups. If $F$ is a finite subgroup then the Mackey formula shows that the restriction of $N$ to $F$ is of the form $\Coind^F_1 Z$. 	But for finite groups, coinduction is equivalent to induction, so we can use the result for induced modules just proved.
	\end{proof}

Now for the proof of Theorem~\ref{thm:compres}. In \cite{GG}, two invariants are defined:
\[
\spli (kG) = \sup{} \{ \projdim_{kG} I \mid I \mbox{ injective} \}\quad
\silp (kG) = \sup \{ \injdim_{kG} P \mid P \mbox{ projective} \}.
\]
From Lemma~\ref{lem:coind} we know that $\projdim_{kG} I$ is finite for any injective module $I$. Thus $\projdim_{kG} I$ is bounded by $\findim kG$, which is finite by Lemma~\ref{lem:findim}, since $G$ is of type $\Phi_k$. We conclude that $\spli (kG)$ is finite. Now \cite[4.1]{GG} constructs the acyclic complex of projectives in the definition of a complete resolution (they call this a complete resolution even though they do not require it to be totally acyclic), see also \cite{ikenaga}. In fact, total acyclicity is automatic.

\begin{lemma}
	\label{lem:tac}
	If $G$ is of type $\Phi_k$ then any acyclic complex of modules is totally acyclic.
	\end{lemma}

\begin{proof}
	Let $X_*$ be the acyclic complex. We claim that $\Hom_{kG}(X_*,M)$ is acyclic for $M$ of finite injective dimension. This is clearly true for $M$ injective and the general case is proved by an easy induction on $\injdim_{kG} M$. Thus all we need to know to finish the proof is that $\silp (kG)$ is finite and this is proved in \cite[2.4]{GG}.
	\end{proof}

\begin{remark}
	The proof of \cite[2.4]{GG} is the only place in this section where the argument requires $k$ to be noetherian.
	\end{remark}

Following \cite{benson, ikenaga}, we can define $\wh{\Ext}^i_{kG}(M,N)= \wh{\Hom}_{kG}(\Omega^iM,N)$ and $\wh{H}^i(G;N)=\wh{\Ext}^i_{kG}(k,N)$, for all $i \in \mathbb Z$. When $G$ has finite $\vcd_k$ then this is the same as the Tate-Farrell cohomology in \cite{brown}.

\section{Functors}
\label{sec:functors}

From now on we assume that all groups are of type $\Phi_k$ and we continue to assume that $k$ is noetherian of finite global dimension.

We want to know when and how a functor $F \! : \Mod(kG) \rightarrow \Mod(kH)$ induces a triangulated functor between the stable categories $\smod$. Most, but not all, of the problems only arise when $k$ is not a field. If $F$ is exact and takes projectives to projectives, then clearly it induces a functor of triangulated categories on $\smod$; this is easy to see using any version of $\smod$. In this way we obtain induction $\Ind^G_H$ and restriction $\Res^G_H$ of stable categories for any subgroup $H$ of $G$.

In fact, it suffices for $F$ to be exact and to take projective modules to modules of finite projective dimension. For a short exact sequence   $\Omega M \rightarrow P \rightarrow M$ in $\Mod(kG)$ with $P$ projective is taken to a short exact sequence $F(\Omega M) \rightarrow F(P) \rightarrow F(M)$ in $\Mod(kH)$; but $F(P)=0$ in $\Stab(kH)$, so $F(\Omega M)=\Omega F(M)$ in $\Stab(kH)$. If $f \in \Hom_{\Stab(kG)}(M,N)$ is represented by $f' \! : \Omega^n M \rightarrow \Omega^n N$, then $F(f') \! : \Omega^n F(M) \rightarrow \Omega^n F(N)$ determines an element $F(f) \in \Hom_{\Stab(kH)}(F(M),F(N))$.

In this way we obtain an inflation map $\Inf^{G/N}_G$ when $N$ is a normal subgroup of $G$ such that $G/N$ is of type $\Phi_k$ and any element of finite order in $N$ has order a unit in $k$. For given a projective $k(G/N)$-module $P$ inflate it to $G$; then for any finite subgroup $F \leq G$, $P$ is projective over $k(F/(F \cap N))$ and $|F \cap N|$ is a unit in $k$, so $P$ is projective over $kF$. Thus $P$ is of finite projective dimension over $kF$ so, since $G$ is of type $\Phi_k$, we find that $P$ is of finite projective dimension over $kG$.  

More generally, for any homomorphism $h \! : J \rightarrow K$ between groups of type $\Phi_k$ such that any element of finite order in the kernel has order a unit in $k$, there is a pullback map $h^* \! : \smod (kK) \rightarrow \smod (kJ)$.

\begin{lemma}\label{lem:f1}
	Change of scalars across a homomorphism of rings $f \! : k \rightarrow \ell$ in either direction, i.e.\ restriction $f^*$ or base change $f_*$, induce functors between $\Mod (kG)$ and $\Mod (\ell G)$ that take projective modules to modules of finite projective dimension (the assumptions at the beginning of this section holding for both $k$ and $\ell$). Also $f^*$ is exact.
	\end{lemma}

\begin{proof}
	 By Lemma~\ref{lem:coind}(1), a projective module is a summand of one of the form $\Ind^G_1 X$ for $X$ free. Both $f^*$ and $f_*$ commute with $\Ind^G_1$; $f_*X$ is still free so $\Ind^G_1f_*X$ is free and for $f^*$ use Lemma~\ref{lem:coind}(2). Exactness of $f^*$ is clear.
	 \end{proof}
 
 \begin{lemma}\label{lem:f2}
 	For a given $kG$-module $M$, the following functors from $\Mod(kG)$ to itself take projective modules to modules of finite projective dimension. If $M$ is a lattice then the first two are exact.
 	\begin{enumerate} 
 		\item
 		$-\otimes_kM$,
 		\item
 		$\Hom_k(M,-)$,
 		\item
 		$\Hom_k(-,M)$.
 		\end{enumerate}
 		\end{lemma}
 		
 \begin{proof}
 Any projective module is a summand of $\Ind^G_1X$ for some free $k$-module $X$.
 
 1) It is well known that $(\Ind^G_1X) \otimes_k M \cong \Ind^G_1 (X \otimes_kM)$. Now use Lemma~\ref{lem:coind}.
 
 2) For any finite subgroup $F \leq G$, $\Res^G_F \Hom_k(M,\Ind^G_1X)
 \cong \Hom_k(M,\Ind^F_1Y)$ for some $k$-module $Y$, by the Mackey
 formula. Now,
 $$\Hom_k(M,\Ind^F_1Y) \cong \Hom_k(M,\Coind^F_1Y) \cong \Coind^F_1
 \Hom_k(M,Y),$$
 so we can use Lemma~\ref{lem:coind} to see
 that $\Res^G_F \Hom_k(M,\Ind^G_1X)$ has finite projective dimension.
 Since $G$ is of type $\Phi_k$ we can deduce that $\Hom_k(M,\Ind^G_1X)$ is also of finite projective dimension.
 
 3) $\Hom_k(\Ind^G_1X,M) \cong \Coind^G_1 \Hom_k(X,M)$, so we can use Lemma~\ref{lem:coind} again.
 
 The exactness statement is trivial.
 \end{proof}

The lemmas above together with the discussion at the beginning of this
section show that if $M$ is a lattice then $-\otimes_kM$ and
$\Hom_k(M,-)$ induce triangulated functors of the stable category as
before; so does $f^*$. In the other cases, we see that each of the
functors in Lemmas~\ref{lem:f1} and \ref{lem:f2} naturally defines a functor
$\GPmp(kG) \rightarrow \Stab(kG)$ (with the obvious variation for
$f_*$), because now the triangles are isomorphic to short exact
sequences, and all modules are projective over $k$ so short exact
sequences are split and exactness is automatic. By Theorem~\ref{thm:catequiv}, we can regard this as a functor from $\smod(kG)$ to itself.
This second approach also makes it transparent what the functor does to morphisms, in particular for $\Hom_k(-,M)$.

By Theorem~\ref{thm:catequiv}, the categories $\smod(kG)$ and
$\Stab(kG)$ are equivalent, so we can regard any of these functors as a functor from
$\Stab(kG)$ to itself. However, it might not be given by the usual formula
on objects that are not Gorenstein projective, because we need to
replace a module $M$ by its Gorenstein approximation $\tilde{M}$
before applying the formula (see Definition~\ref{def:g-approx}). 

\begin{example} Let $C$ be a cyclic group of order $p$ and
  $\hat{\mathbb Z}_p$ the $p$-adic integers. Consider the trivial
  module $\mathbb F_p$ for $\hat{\mathbb Z}_pC$. An easy calculation
  (cf.\ \cite[2.6]{S3}) shows that $\tilde{\mathbb F}_p \simeq
  \hat{\mathbb Z}_p \oplus \Omega \hat{\mathbb Z}_p$, hence $
  \tilde{\mathbb F}_p  \otimes_{\hat{\mathbb Z}_p} \tilde{\mathbb F}_p
  \simeq (\hat{\mathbb Z}_p)^2 \oplus (\Omega \hat{\mathbb
    Z}_p)^2$. By definition, $\mathbb F_p \otimes_{\hat{\mathbb Z}_p}
  \mathbb F_p \simeq \tilde{\mathbb F}_p  \otimes_{\hat{\mathbb Z}_p}
  \tilde{\mathbb F}_p \simeq ({\mathbb F}_p)^2$, which is not what
  might have been expected.  
	\end{example}

For any $H \leq G$, coinduction $\Coind^G_H$ also induces a functor of
stable categories. For a projective $kH$-module $P$ and any finite
subgroup $F \leq G$, we have
$$\Res^G_F \Coind^G_H P = \prod_{g \in F
  \backslash G/H} \Coind^F_{F \cap {}^gH} {}^g \Res^H_{F^g \cap H}P.$$
Since coinduction is equivalent to induction for finite groups,
the factors in the product are clearly projective, thus so is the
product, by Lemma~\ref{lem:prod} below. It follows that $\Coind^G_H P$
must be projective.

\begin{lemma}
	\label{lem:prod}
	If the $kG$-modules $P_i$, $i \in I$, are projective then the product $\prod _{i \in I} P_i$ is of finite projective dimension.
	\end{lemma}

\begin{proof}
	Because $G$ is of type $\Phi_k$, it suffices to prove the case when $G$ is finite. 
	For each $i \in I$ we can write $P_{i}$ as a summand of
        $\Ind^G_1 X_i$ for some free $k$-module $X_i$. Thus $\prod _{i
          \in I} P_i$ is a summand of $\prod _{i \in I} \Ind^G_1
        X_i$. But since $G$ is finite, $\Ind^G_1 X_i$ is isomorphic to
        $\Coind^G_1 X_i$, and $\Coind^G_1$, being a right adjoint,
        commutes with products. Hence $\prod _{i \in I} \Ind^G_1 X_i
        \cong \Ind^G_1 \prod _{i \in I}  X_i$, which is projective. 
	\end{proof}

It is easy to check that $\Ind^G_H$ is still the left adjoint and $\Coind^G_H$ is still the right adjoint of restriction on the stable categories.

\begin{prop}
	$\smod(kG)$ has products and coproducts, i.e.,
 \begin{align*}
\textstyle	\prod_i \shom ( A_i,B) & \cong \shom ( \oplus_i A_i, B) \\
\textstyle	\prod_j \shom (A, B_j) & \cong \shom (A, \textstyle \prod_j B_j),
	\end{align*}
	where $\prod$ and $\oplus$ are the usual ones in $\Mod$.
	\end{prop}

\begin{proof}
	We may assume that all the modules are Gorenstein projective,
        so we are working in $\modproj(kG)$. The only part that is not
        entirely routine is to check that if each $f_i \! : A_i
        \rightarrow B$ factors through a projective $P_i$ then the
        induced map $\oplus_i A_i \rightarrow B$ factors through a
        projective. But it factors through $\oplus_i P_i$, which is
        projective.  
	
	If the $g_j \! : A \rightarrow B_j$ each factor through a
        projective $Q_j$ then the induced map $A \rightarrow \prod_j
        B_j$ factors through $\prod_j Q_j$. This has finite projective
        dimension, by Lemma~\ref{lem:prod}, so the induced map is 0 in
        $\smod(kG)$. 
	\end{proof}
	 
\section{Decompositions}\label{sec:decomp}

The Eckmann-Hilton argument, on sets with two monoid structures (also
called the Eckmann-Hilton theorem, see
\cite{EH}), applies and gives us the following fundamental result.

\begin{prop}\label{prop:end-comm}
Let $G$ be a group of type $\Phi_k$. The ring
$\send_{kG}(k)$ and the group $\saut_{kG}(k)$ are 
commutative. The product under composition agrees with the product
under tensor product.
\end{prop}

Unlike ${\End}_{kG}(k)$, the $k$-algebra $\send_{kG}(k)$ can be quite complicated, as we shall see.

The next lemma is very useful.

\begin{lemma}
	\label{lem:iso}
	Let $G$ be a group of type $\Phi_k$ and let $f \! : M \rightarrow N$ be a morphism in $\smod(kG)$ that restricts to a stable isomorphism on any finite subgroup. Then $f$ is a stable isomorphism.
	\end{lemma}

\begin{proof} Consider the cone of $f$. It is stably 0 on restriction to any finite subgroup, so by the definition of type $\Phi_k$ it is stably 0 over $G$.
	\end{proof}

\begin{prop}
	Let $G$ be a finite group and let $\Delta$ be a
        $G$-CW-complex. Suppose that the reduced homology groups $\tilde{H}_i(\Delta^H;k)=0$ for $i \geq 0$ and any non-trivial subgroup $H \leq G$. 
	\begin{enumerate}
		\item
		The augmented cellular chain complex $\tilde{C}_*(\Delta)$ is
                chain homotopy equivalent to a bounded-below complex
                of projective $kG$-modules.
              \item If
                $H_i(\Delta;k)=0$ for large enough $i$ then
                $\tilde{C}_*(\Delta)$ is chain homotopy equivalent to
                a bounded complex of projective $kG$-modules. 
		\end{enumerate}
	\end{prop}

The original version of this result is due to Webb \cite{webb} for $k$ a complete $p$-local ring and $\Delta$ finite dimensional. There is another proof of the first part by Bouc \cite{bouc} for $k$ a field and general $\Delta$.

\begin{proof}
	The first part is proved in \cite{Sy2}; it is the case $H=1$ of the statement that $q$ is a homotopy equivalence that appears just before the lemma (there is no assumption on $k$; see the remark at the end). There is a similar proof in \cite[6.6]{Sy1}, but note that in 6.4 and 6.6 there the word bounded should not appear unless $\Delta$ is finite dimensional. 
	
	For the second part, let $C_*=\tilde{C}_*(\Delta)$ and let $P_*$ be the bounded below complex of projectives. Let $\tau_{< n}C_*$ be a good truncation of $C_*$, truncated in some degree $n$ greater than the degree of any non-zero homology group \cite[1.2.7]{weibel}. Then we have a quasi-isomorphism $P_* \rightarrow \tau_{< n}C_*$. By \cite[6.5]{Sy1}, $P_*$ is a summand of a complex that is split in high degrees, so it is split in high degrees itself. Thus $P_*$ is homotopy equivalent to a good truncation of itself.
	\end{proof}

\begin{thm}
	Let $G$ be a group of type $\Phi_k$ and let $\Delta$ be a $G$-CW-complex such that $H_i(\Delta;k)=0$ for large enough $i$. Suppose that $\tilde{H}_i(\Delta^H;k)=0$ for $i \geq 0$ and any non-trivial finite subgroup $H \leq G$. Then the chain complex ${C}(\Delta)$, considered as an element of $D^b(\Mod(kG))/K^b(\Proj(kG))$, is equal to $k$. The same is true if $k$ is $p$-local and we only require $H_i(\Delta^H;k)=0$ for non-trivial $p$-groups $H$. In these cases $k$ decomposes as a direct sum of non-zero  pieces corresponding to those path components of $\Delta/G$ for which some cell of $\Delta$ above them is fixed by an element of finite order not a unit in $k$.
	\end{thm}

\begin{proof}
	By the previous proposition, for any finite subgroup $F$ ($p$-subgroup if $k$ a $p$-local ring), the restriction of the augmented chain complex is 0 in $D^b(\Mod(kF))/K^b(\Proj(kF))$. By Lemma~\ref{lem:iso}, the augmented chain complex is 0 in $D^b(\Mod(kG))/K^b(\Proj(kG))$. Thus the augmentation is an isomorphism from ${C}(\Delta)$ to $k$.
	
	Let $X$ be the part of $\Delta$ corresponding to a given component of $\Delta/G$. If $X^{\langle g \rangle} \ne \emptyset$ for some $g$ then the augmentation of $C(X)$ is split over $\langle g \rangle$. Since $k \ne 0$ over $k\langle g \rangle$ when $|g|$ is finite but not a unit (because $\hat{H}^0(\langle g \rangle;k)=k/|g|k$), we see that $C(X) \ne 0$. If no such $g$ exists, then the restriction of $C(X)$ to any finite subgroup $F$ is a complex of projectives. By the proof of the previous proposition, it is homotopy equivalent to a bounded complex of projectives, so is 0 in $D^b(\Mod(kF))/K^b(\Proj(kF))$ for all $F$, hence 0 in $D^b(\Mod(kG))/K^b(\Proj(kG))$, since $G$ is of type $\Phi_k$.
	\end{proof}

An obvious candidate for $\Delta$ when $k$ is $p$-local is the Quillen complex $\Delta (\mathcal S_p(G))$, the simplicial complex constructed from chains of non-trivial $p$-subgroups, or the Brown complex $\Delta (\mathcal A_p(G))$, the simplicial complex constructed from chains of non-trivial elementary abelian $p$-subgroups. These are known to satisfy the condition on fixed point sets and in fact they are equivariantly homotopy equivalent. The latter is finite dimensional if the $p$-rank of $G$ is finite, so in this case they both satisfy the condition on their homology groups.

For general $k$ we can always use $\Delta(\mathcal F(G))$, constructed from chains of non-trivial finite subgroups.

In the $p$-local case, this decomposition is the best possible. Let $I$ index the components of $\Delta(\mathcal S_p(G))/G$ and let $k_{e_i}$ denote the part of $C(\Delta(\mathcal S_p(G)))$ corresponding to component $i$. Then $k \simeq \oplus _i k_{e_i}$. Let $e_i \in \send_{kG}(k)$ be the idempotent corresponding to projection onto $k_{e_i}$; then $\send_{kG}(k) = \prod \send_{kG}(k_{e_i})=\prod e_i \send_{kG}(k)$.

\begin{thm}
	When $k$ is $p$-local, the primitive idempotents of $\send_{kG}(k)$ correspond to the path components of $\Delta(\mathcal S_p(G))/G$, or equivalently to the equivalence classes of non-trivial $p$-subgroups of $G$ under the equivalence relation generated by inclusion and conjugation.
	\end{thm}

\begin{proof}
	Let $X_i$ be the part of $\Delta(\mathcal S_p(G))$ corresponding to component $i \in I$. If $P$ is a $p$-subgroup that appears in this part of $\mathcal S_p(G)$ then $X_i^P \ne \emptyset$, so the augmentation $C(X_i) \rightarrow k$ is split over $kP$ and so $k \downarrow _P \mid C(X_i)\downarrow_P$. But $e_i$ acts as the identity on $C(X_i)$, so $e\downarrow_P$ acts as the identity on $k\downarrow_P$, i.e.\ $\rest^G_P(e_i)=1$. The idempotents are orthogonal, so $\rest^G_P(e_j)=0$ for $j \ne i$. 
	
	Suppose that $e=e_i$ decomposes in $\send_{kG}(k)$ as a sum of idempotents, $e=f_1 + f_2$. For any finite $p$-group $P$, we have $\send_{kP}(k)= k/|P|k$, which is still local if $P \ne 1$, so the only idempotents are 0 and 1. Thus $\rest_P(f_1)$ is 0 or 1. This choice is preserved by restriction and conjugation, so it is constant on all subgroups $P$ in the component $i$; say $\rest_P(f_1)=1$ and $\rest_P(f_2)=0$. Thus $f_1 \! : k_e \rightarrow k_e$ is an isomorphism on restriction to any subgroup in $i$ and also on restriction to any subgroup in any other component because then the restriction of $k_e$ is 0. By Lemma~\ref{lem:iso}, $f_1$ is an automorphism and so, being an idempotent, $f_1=e$ and thus $f_2=0$.
	\end{proof} 

A version of this result was obtained for $\hf 1$ groups by Cornick
and Leary \cite{CL}; this class includes groups of unbounded
$p$-rank. They use the complex that appears in the definition of $\hf 1$ as $\Delta$; the condition on the fixed point spaces of finite $p$-groups is satisfied, by Smith Theory. It is worth noting that, by work of
Freyd \cite{freyd}, idempotents in $\send_{kG}(k)$ split $k$ stably;
our construction avoids quoting this.

\begin{example}\label{ex:cp*cp}
Let $G=A*B$ be the free product of two groups $A$ and
$B$, both of type $\Phi_k$.
We know that $G$ is of type $\Phi_k$ and that
any non-trivial finite subgroup of $G$ is conjugate to a subgroup of $A$ or $B$ but not both \cite[II.A3]{brown}.  
Thus
there are at least two components provided both $A$ and $B$ contain elements of order not a unit in $k$. 

Since $\send_{kG}(k) \cong \hat{H}^0(G;k)$, we can use the complete cohomology version of \cite[VI.3]{brown} to see that $\send_{kG}(k) \cong \send_{kA}(k) \times \send_{kB}(k)$. If both $A$ and $B$ are finite then $\send_{kG}(k) \cong k/|A|k \times k/|B|k$.

This method of calculation can be generalised to any group of type $\Phi_k$ that acts on a tree.
\end{example}

\section{Invertible modules}

Throughout this section, $G$ denotes a group of type $\Phi_k$.

\begin{defn}
	\label{def:invertible}
A $kG$-module $M$ is {\em invertible} if there exists a $kG$-module $N$
such that $M\otimes N\cong k$ in $\wh{\Mod}(kG)$. 
The tensor product $-\otimes-=-\otimes_k-$ of Section~\ref{sec:functors} equips the set of isomorphism
classes of invertible $kG$-modules in $\wh{\Mod}(kG)$ with the structure
of an abelian group, which we denote by $T_k(G)$ and call the {\em group of
  invertible $kG$-modules}.
\end{defn}

For finite groups, the endotrivial modules are defined in the same way, except that the modules are required to be finitely generated.

If we take $M$ and $N$ to be Gorenstein projective then we have $M \otimes N \cong \tilde{k}$ in $\GPmp(kG)$, with the usual tensor product (recall that the tensor product of two Gorenstein projective modules is Gorenstein projective since $\gldim k < \infty$). 

First we check that for finite groups the basic theory of endotrivial modules carries over to our context, where $k$ need not be a field and the modules need not be finitely generated. Note that for a finite group a module being Gorenstein projective is equivalent to it being a lattice and if a module is finitely generated then it has a finitely generated Gorenstein projective approximation. This is because a sufficiently high syzygy is finitely generated (since $k$ is noetherian) and Gorenstein projective; it is a kernel in a complete resolution with finitely generated terms by \cite[VI.2.6]{brown}.

The following result is stated for $k$ a field in \cite[2.1]{BBC}).

\begin{thm}\label{thm:BBC} 
Let $G$ be a finite group and $M$ an
  invertible $kG$-module.
Then $M$ is stably a summand of a finitely generated $kG$-module and the natural map $M \rightarrow M^{**}$ is a stable isomorphism. If $k$ is a complete local ring and $M$ is a lattice then $M$ decomposes as a $kG$-module as a direct sum $M'\oplus P$ with $M'$ a finitely generated lattice and $P$ projective.
\end{thm}

\begin{proof}
We may assume that $M$ is Gorenstein projective. Let $f \! : k \rightarrow M \otimes N$ be a stable isomorphism; since $k$ is a lattice, $f$ is a genuine homomorphism.  Write $f(1) = \sum_i m_i \otimes n_i$ for $m_i \in M$, $n_i \in N$, and let $M'$ be the $kG$-submodule of $M$ generated by the $m_i$. We have a stable isomorphism $k \stackrel{f}{\rightarrow} f(k) \leq M' \otimes N \subseteq M \otimes N \stackrel{f^{-1}}{\rightarrow} k$, so $k \mid M' \otimes N$ stably. It follows that stably $M \mid M' \otimes N \otimes M \simeq M'$.

Let $L=\tilde{M'}$, so $M \mid L$ stably. $L$ is a finitely generated lattice, so the natural map $L \rightarrow L^{**}$ is a stable isomorphism. This property is inherited by stable summands, in particular $M$. 

Somehow we need to deduce the last part. The authors of \cite{BBC} probably had in mind the method of Rickard \cite[3.2]{rickard}, but this is for $k$ a field; instead we use the Crawley-J{\o}nsson-Warfield Theorem \cite[26.6]{AF}, which is a version of the Krull-Schmidt Theorem that applies in this case. Since both $L$ and $M$ are Gorenstein projective, we know that $M$ is a summand as a module of some $U=L \oplus (\oplus_{i \in I} kG)$. Both $L$ and $kG$ are finitely generated and $k$ is noetherian, so they can both be expressed as a finite sum of finitely generated indecomposables and thus so can $U$, say $U = \oplus_{j \in J} X_j$, with only finitely many $X_j$ not projective. By \cite[6.10]{CR}, since $k$ is a complete local ring, each of these indecomposables has a local endomorphism ring and is clearly countably generated, so the Crawley-J{\o}nsson-Warfield Theorem tells us that $M$ is isomorphic to  $\oplus_{j \in K} X_j$ for some $K \subseteq J$.
\end{proof} 

This theorem, together with Lemma~\ref{lem:stab-iso}, yields the following result.

\begin{cor} For a finite group  $G$ and a complete local (noetherian) ring $k$ we get the same group $T_k(G)$ whether we consider invertible modules in $\smod(kG)$, as we do in this paper, or the classical group of endotrivial modules in $\underline{\mod}(kG)$.
	\end{cor}

The next result is well known when $M$ is finitely generated and $k$ is a field. Note that we do not assume that the trivial module $k$ is stably indecomposable, so we cannot assume that an invertible module is stably indecomposable (a decomposition can occur when $k= \mathbb Z$, for example).

\begin{thm}
	Let $G$ be a finite group and let $M$ and $N$ be $kG$-modules such that $M \otimes N \simeq k$.
	Then:
	\begin{enumerate}
		\item $N \simeq M^*$,
		\item the evaluation map $\ev \! :M \otimes M^* \rightarrow k$ is a stable isomorphism and
		\item the natural map $\theta_M \! : M\otimes M^* \rightarrow \End_k(M)$ is a stable isomorphism.
		\end{enumerate}
	\end{thm}

\begin{proof}
	We can assume that $M$ and $N$ are Gorenstein projective and,
        by the previous theorem, stably isomorphic to their double
        duals. From the stable isomorphism $f \! :M \otimes N
        \rightarrow k$, realised as a genuine homomorphism,  we obtain
        a map $f' \! : N \rightarrow M^*$ such that $f = \ev ( \id_M
        \otimes f')$. Thus $k \mid M \otimes M^*$ stably and so $N
        \mid N \otimes M \otimes M^* \simeq M^*$ stably. By symmetry,
        $M \mid N^*$ stably; say $N \simeq M^* \oplus X$ and $M \simeq
        N^* \oplus Y$. Eliminating $N$ and using the fact that $M
        \simeq M^{**}$, we obtain $M \simeq M \oplus X^* \oplus Y$. 
	
	By Theorem~\ref{thm:BBC}, there is a finitely generated
        lattice $L$ such that $M \mid L$ stably. Applying
        $\shom_{kG}(L,-)$ to the equation, we obtain $\shom_{kG}(L,M)
        \cong \shom_{kG}(L,M) \oplus \shom_{kG}(L,X^*)
        \oplus\shom_{kG}(L,Y)$. These are finitely generated
        $k$-modules and $k$ is noetherian, so we must have
        $\shom_{kG}(L,Y)=0$. But $Y \mid M$ stably, hence $Y \mid L$
        stably, so $\send_{kG}(Y)=0$; thus $Y \simeq 0$. By symmetry, $X
        \simeq 0$, yielding $N \simeq M^*$. 
	
	The equation $f = \ev ( \id_M \otimes f')$ shows that $\ev$ is split, so $k$ is stably a summand of $M \otimes M^*$, say $M \otimes M^* \simeq k \oplus Z$. But we now know that $M \otimes M^* \simeq M \otimes N \simeq k$, so we have $k \simeq k \oplus Z$. As before, applying $\shom_{kG}(k,-)$ shows that $Z \simeq 0$.
	
	For the final part, $L$ is finitely generated, so we know that
        the natural map $\theta_L \! : L\otimes
        L^*\rightarrow \End_k(L)$ is a stable isomorphism. But $L
        \simeq M \oplus U$ for some $U$, so we can write $L \otimes
        L^* \simeq (M \otimes M^*) \oplus (M \otimes U^* \oplus U
        \otimes M^* \oplus U \otimes U^*)$ and $\End_k(L) \simeq
        (\End_k(M)) \oplus (\Hom_k(U,M) \oplus \Hom_k(M,U)
        \oplus \End_k(U))$, both considered as a sum of two
        submodules. By definition, $\theta_L$ respects these decompositions and
        restricts to $\theta_M$ on $M \otimes M^*$. In other words, the direct sum of $\theta_M$ and another morphism is a stable isomorphism; it follows that $\theta_M$ is a stable isomorphism. 
	\end{proof}
	
Now for infinite groups of type $\Phi_k$.

\begin{thm}\label{thm:invertible1}
A $kG$-module $M$ is invertible if and only if 
$M\res GF$ is invertible for every finite subgroup $F$ of $G$.
Moreover, if $M$ is invertible, then $M^*$ is an inverse for $M$.
\end{thm}

Note that if $k$ is $p$-local then the condition that $M\res GF$ be invertible for every finite
subgroup $F$ of $G$ is satisfied if and only if it is satisfied when $F$ runs through
a set of representatives of the conjugacy classes of finite elementary
abelian $p$-subgroups of $G$ (\cite{CMN}).

\begin{proof}
	Clearly, if $M$ is invertible, then so is $M\res GF$.
	For the converse, consider the map $\ev \! : M \otimes M^* \rightarrow k$. The previous theorem shows that for finite groups this is a stable isomorphism if and only if $M$ is invertible. Thus if $M$ is invertible on all finite subgroups then $\ev$ is a stable isomorphism on all finite subgroups, hence a stable isomorphism over $kG$, by Lemma~\ref{lem:iso}.
	\end{proof} 

A similar proof works for the next result.

\begin{prop}
	If $M$ is an invertible $kG$-module then $\theta_M \! : M \otimes M^* \rightarrow \End_k(M)$ is a stable isomorphism, as is the natural map $M \rightarrow M^{**}$.
	\end{prop}

Here is another result from the theory of endotrivial modules for finite groups
which extends to invertible modules for groups of type $\Phi_k$.

\begin{cor}
Let $M$ be a $kG$-module. 
For any $n\in\Z$, if $M \otimes N \simeq k$ then $\Omega^nM \otimes \Omega^{-n}N \simeq k$.
\end{cor}

\begin{proof}
The assertion holds for all finite subgroups of $G$ (see
\cite[Section 2]{CMN}) and therefore for
$G$ too.
\end{proof}

\begin{lemma}\label{lem:tg-op}
Let $H$ be a subgroup of $G$ and let $N$ be a normal subgroup such that any element of finite order has order a unit in $k$ and the quotient group $G/N$ is of type
$\Phi_k$. Restriction 
and inflation 
induce group homomorphisms
$$\Res^G_H \! : T(G)\to T(H)\qbox{and}
\Inf^{G/N}_G \! : T(G/N)\to T(G),$$
which commute with $\Omega$. More generally, there is a pullback map
$h^* \! : T(K) \rightarrow T(J)$ for any homomorphism $h \! : J
\rightarrow K$ such that any element of finite order in the kernel has
order a unit in $k$. 
There is also a base change map $f_* \! : T_k(G) \rightarrow T_{\ell}(G)$ for any $f \! : k \rightarrow \ell$ for suitable $\ell$.
\end{lemma}

\begin{proof}
	This follows immediately from the fact that the functors
        commute with tensor product in $\smod$. However, given the way
        we have defined functors on $\smod$ in
        Section~\ref{sec:functors}, the latter is not quite
        obvious. In the case of $\Inf$, for example, we need to check
        that $\Inf^{G/N}_G M \otimes \Inf^{G/N}_G N \simeq
        \Inf^{G/N}_G (M \otimes N)$.
        By taking Gorenstein approximations
        (Definition~\ref{def:g-approx}), we can assume that $M$ and $N$ are
        Gorenstein projective, hence lattices; the same will be true
        of their inflations. By the discussion at the beginning of
        Section~\ref{sec:functors},  both $\Inf^{G/N}_G(-)$ and
        $-\otimes_k-$ applied to these modules have their usual
        meaning, and therefore commute.
	\end{proof}
	
\begin{lemma}
	\label{lem:inf}
	If $G$ and $N$ are as above and the quotient map $\pi \! : G \rightarrow G/N$ is split then $\Inf^{G/N}_G$ is split.
	\end{lemma}

\begin{proof}
	Let $\sigma$ be a splitting homomorphism, so $\pi \sigma = \id$. Then $ \sigma^* \Inf^{G/N}_G = \sigma^* \pi^* = \id$.
	\end{proof}

Another source of invertible modules comes from the observation that any $kG$-module that is a free $k$-module of rank 1 is invertible. 
 
\begin{prop}
	\label{pr:lift}
	Suppose that $G$ is finite and $k$ is a complete discrete valuation ring with residue class field $\overline{k}$. Then there is a short exact sequence $$0 \rightarrow \Hom(G, \tor_p(k^{\times})) \rightarrow T_k(G) \rightarrow T_{\overline{k} }(G) \rightarrow 0,$$ which is split. Here $\tor_p$ denotes the $p$-torsion subgroup and we are identifying $\Hom(G, k^{\times})$ with the rank 1 lattices.
	\end{prop}

\begin{proof} By Theorem~\ref{thm:BBC}, we may take the modules to be finite rank lattices. Surjectivity at $T_{\overline{k}} (G)$ is proved in \cite[1.3]{LMS} in the case when $k$ has characteristic 0. When it has characteristic $p$ then $\overline{k}$ lifts to a subring of $k$ (Cohen's Structure Theorem, see e.g.\ \cite[28.3]{mat}), so modules lift too.
	
	If $M$ is a $kG$-lattice such that $\overline{k} \otimes_kM \cong \overline{k} \oplus \proj$ as $\overline{k}G$-modules, then the projective summand lifts to a projective summand of $M$. The complement is of rank 1 and thus corresponds to an element of $\Hom(G, \tor_p(k^{\times}))$; this proves exactness in the middle. For injectivity on the left, note that if two rank 1 lattices are isomorphic modulo projectives then they are isomorphic, by the Krull-Schmidt theorem, unless $p$ does not divide $|G|$, in which case the proposition is trivial.
	
	For the splitting, first note that if $k$ has characteristic $p$ then $\tor_p(k^{\times})=1$ and there is nothing to prove, so we may assume that $k$ has characteristic 0. Let $\ell \leq k$ be a coefficient ring (Cohen's Structure Theorem again, \cite[29.3]{mat}), so $\ell$ is a complete local ring with maximal ideal generated by $p$, and also $\overline{\ell} \cong \overline{k}$. Thus $\ell$ contains no $p$th roots of unity and the first part of the proposition tells us that $T_{\overline{k} }(G) \cong T_{\overline{\ell} }(G) \cong T_{\ell }(G)$. The base change map $T_{\ell }(G) \rightarrow T_{k}(G)$ of Lemma~\ref{lem:tg-op} now completes the splitting. 
	\end{proof}

\section{Groups Acting on Trees}\label{sec:pi1graph}

In this section, we describe $T_k(G)$ for groups $G$ of type $\Phi_k$ which
act on  trees, in other words for the fundamental group of a graph of groups. We follow \cite{DD} for the notation and
conventions. As special cases, we will obtain $T_k(G)$ for amalgamated
free products and HNN extensions of groups of type $\Phi_k$.

Suppose that a group
$G$ acts on a tree $T$ with vertex set $VT$ and edge set $ET$. If all the vertex
stabilisers are of type $\Phi_k$ (hence the edge stabilisers too) and there exists a common bound on
their finistic dimensions, then $G$ is of type $\Phi_k$ by Proposition~\ref{prop:g-cw-cx}.
 In particular, this applies when the stabilisers are finite.

We continue to impose our assumptions on $k$ from the beginning of
Section~\ref{sec:phi} , and all groups will be of type $\Phi_k$ except
where indicated (and of course, this does not apply to groups such as
$T_k(G)$).
We will omit the subscript $k$ on $T$ and $\Phi$ and write $\saut_G(k)$ for $\saut_{kG}(k)$.

The edges of $T$ are oriented, so each edge $e\in ET$ has an initial vertex $\iota(e)$ and a terminal
vertex $\tau(e)$. We pick a fundamental $G$-transversal
$Y$ of $T$ and a maximal subtree $Y_0$ of $Y$. (So $Y$ is a fundamental
domain for the action of $G$ on $T$ and $VY=VY_0$; an edge in
$Y_0$ has both vertices in $Y_0$, but an edge in $Y\smallsetminus Y_0$
has only its initial vertex in $Y_0$.)
For each vertex $v\in VT$, let $\bar v$ be the unique vertex of $Y_0$ in the same $G$-orbit
as $v$ and choose $t_v\in G$ such that $t_v\bar v=v$, with $t_v=1$ if
$\bar v=v$.

 Let $G_v$ and $G_e$ denote the
 stabilisers in $G$ of $v\in VT$ and $e\in ET$, so
$G_e=G_{i(e)}\cap G_{t(e)}$.
For each edge $v \in VT$  there is an isomorphism
$f_v:G_v\to G_{\bar v}$, given by $g\mapsto t_v^{-1}gt_v$.

Suppose that for each $v\in VY$, we are given a Gorenstein projective
$kG_v$-module $M_v$, and for each edge $e\in EY$ a stable
$kG_e$-isomorphism (which we can assume to be a genuine morphism, by
Lemma~\ref{lem:nat-onto})
$$\varphi_e \! : M_{\iota (e)}\res{G_{\iota(e)}}{G_e}\longrightarrow\big(t_{\tau (e)}\otimes M_{\ovl{\tau(e)}}\big)
\res{G_{\tau(e)}}{G_e},$$
where $t_{\tau(e)}\otimes M_{\ovl{\tau(e)}}$ is considered as a
$kG_{\tau(e)}$-module in the same way as it would be in the restriction of $kG \otimes _{kG_{\ovl{\tau(e)}}} M_{kG_{\ovl{\tau(e)}}}$, i.e.\  
$g(t_{\tau(e)}\otimes m)=t_{\tau(e)}\otimes (t_{\tau(e)}^{-1} g t_{\tau(e)})m$ for $g \in G_{\tau(e)}$. Thus $t_{\tau(e)}\otimes M_{\ovl{\tau(e)}} \simeq f^*_{\tau(e)} M_{\ovl{\tau(e)}}$. We refer to this data
as $(\ul M,\ul{\varphi})$.

Now, for $v\in VT$, set $M_v=t_v\otimes M_{\bar v}$, a $kG_v$-module, and for $e\in ET$,
$M_e=M_{\iota(e)}\res{G_{\iota(e)}}{G_e}$. Define
$$\varphi_e:M_e\longrightarrow M_{\tau(e)}\res{G_{\tau(e)}}{G_e}\qbox{by}
\varphi_e(t_{\iota(e)}\otimes m)=t_{\tau(e)}t_{\ovl{\tau(e)}}^{-1}\varphi_e(m).$$
Set
$$C_0=\bigoplus_{v\in VT}M_v\qbox{and}C_1=\bigoplus_{e\in ET}M_e,
$$
equipped with the canonical $G$-action, so that
$$C_0\cong\bigoplus_{v\in VY}M_v\ind{G_v}G\qbox{and}
C_1\cong\bigoplus_{e\in EY}M_e\ind{G_e}G.$$

Define a $kG$-homomorphism
$d:C_1\to C_0$ as follows:
for $e\in ET$ and $m\in M_e=M_{i(e)}\res{G_{i(e)}}{G_e}$, let
$dm$ be the difference of $m \in M_{\iota(e)}$ and $\varphi_e(m)$. Let $C(\ul M,\ul{\varphi})$ be the two-term complex
$\xymatrix{C_1\ar[r]^d&C_0}$ and define $\ovl{C}(\ul M,\ul{\varphi})$
to be its cone.

Note that this description of $d$ requires us to work with modules and
genuine homomorphisms. A construction entirely in the stable category
would proceed by defining maps $M_e\to C_0\res G{G_e}$ for $e\in EY$
and then inducing.

This construction gives canonical maps
$\alpha_v:M_v\longrightarrow\ovl{C}(\ul M,\ul{\varphi})\res{G}{G_v}$.

\begin{lemma}\label{lem1}
The following hold.
\begin{enumerate}
\item For each $v\in VT$, the map $\alpha_v$ is a stable $kG_v$-isomorphism.
\item For each $e\in ET$, we have
$\big(\alpha_{\tau(e)}\res{G_{\tau(e)}}{G_e}\big)^{-1}
\big(\alpha_{\iota(e)}\res{G_{\iota(e)}}{G_e}\big)=\varphi_e$ in $\smod(kG_e)$.
\end{enumerate}  
\end{lemma}

\begin{proof}
For (1) it suffices to show that the composite map
$\bar d:C_1\stackrel{d}{\longrightarrow}C_0\longrightarrow C_0/M_w$
is a $kG_w$-isomorphism for all $w\in VY$. We shall construct an
inverse $s$.
Write $(C_0/M_w)\res G{G_w}=\dst\bigoplus_{i=1}^{\infty}C_0^i$, 
where $C_0^i$ is the sum of all the $M_v$ for $v$ at a distance $i$ from $w$.
We shall define $s$ inductively on the $C_0^i$ as follows.
Suppose that we have already defined $s$ on
$\dst\bigoplus_{i=1}^jC_0^i$ for some $j\geq 0$. Given a vertex $v$ at
a distance $j+1$ from $w$, there is a unique edge $e$ with one vertex 
$v$ and the other at a distance $j$ from $w$.
Let $m\in M_v$: if $v=\iota(e)$ then define
$s(m)$ to be the sum of $m \in M_e = M_v\res{G_v}{G_e}$ and $s\varphi_e(m)$; if $v=\tau(e)$ define
$s(m)$ to be the sum of $-\varphi_e^{-1}(m) \in M_e=M_{\iota(e)}\res{G_{\iota(e)}}{G_e}$ and $s(\varphi_e^{-1}(m))$. Here $\varphi_e^{-1}$ denotes a genuine homomorphism that is a stable inverse to $\varphi_e$.

It is straightforward to check that $s$ is a $kG_w$-homomorphism
that is a stable inverse to $\bar d$.

For (2), note that, by construction, we have
$\alpha_{i(e)}(m)-\alpha_{t(e)}(\varphi_e(m))\in\im(d)$ for all
$m\in M_v$ and all $v\in VT$. The equality follows.
\end{proof}

Suppose that each $M_v$ is stably trivial.
Let $\tilde k$ be the Gorenstein projective approximation of the
trivial $kG$-module $k$ and choose stable isomorphisms
$\theta_v:\tilde k\res G{G_v}\longrightarrow M_v$ for all
$v\in VY$. We extend the definition of these maps to all of $VT$ by setting
$\theta_v=t_v\theta_{\bar v}t_v^{-1}$ for $v\in VT$.

\begin{lemma}\label{lem2}
The assignment 
$\varphi_e\mapsto\varphi_e'=\theta_{\tau(e)}^{-1}\varphi_e\theta_{\iota(e)}$
identifies the stable classes of isomophisms
$M_{\iota(e)}\res{G_{\iota(e)}}{G_e}\to M_{\tau(e)}\res{G_{\tau(e)}}{G_e}$ with the
elements of $\wh{\Aut}_{G_e}(k)$.
\end{lemma}

\begin{proof}
This is immediate from the definition of the maps $\theta_v$ and 
$\varphi_e$.
\end{proof}

As a direct consequence of Lemmas~\ref{lem1} and~\ref{lem2}, we obtain
the following.

\begin{cor}\label{cor1}
For every edge $e\in ET$, 
$\varphi_e'=\theta_{t(e)}^{-1}\alpha_{t(e)}^{-1}\alpha_{i(e)}\theta_{i(e)}$ in $\smod(kG_e)$.
\end{cor}

Now take $M_v=\tilde k\res G{G_v}$ and $\theta_v$ to be the identity
map.
Define
\begin{equation}\label{eq:partial}
\partial \! : \prod_{e\in EY}\wh{\Aut}_{G_e}(k)\longrightarrow T(G)
\qbox{by}\partial\big((\varphi'_e)_{e\in EY}\big)=\ovl C(\ul k,\ul{\varphi}),
\end{equation}
where $\ovl C(\ul k,\ul{\varphi})$ is the cone defined above.
This yields an invertible module by Lemma~\ref{lem1} and
Theorem~\ref{thm:invertible1}, since every finite subgroup must fix
some vertex, 
but we do not yet know that $\partial$ is a group homomorphism.

\begin{thm}\label{thm:trees}
Let $G$ be a group acting on a tree with vertex set $VT$ and edge set
$ET$, and assume that all the vertex stabilisers are groups of type
$\Phi_k$ and that there exists a common bound on their finistic dimensions.
Then $G$ is of type $\Phi_k$ and there is an  exact sequence of abelian groups
\begin{multline*}
\wh{\Aut}_{G}(k)\xrightarrow{\Res}
\prod_{v\in VY}\wh{\Aut}_{G_v}(k)\xrightarrow{\Res-\Res_f}
\prod_{e\in EY}\wh{\Aut}_{G_e}(k)\xrightarrow{\partial}...\\
...\xrightarrow{\partial}T(G)\xrightarrow{\Res}
\prod_{v\in VY}T(G_v)\xrightarrow{\Res-\Res_f}
\prod_{e\in EY}T(G_e).\end{multline*}
Here, for $M\in T(G)$, the $T(G_v)$-coordinate of $\Res(M)$ is
$M\res G{G_v}$; for $M_v\in T(G_v)$, $$\qbox{the $T(G_e)$-coordinate of
$\Res(M_v)$ is}\left\{\begin{array}{ll}
M_v\res{G_v}{G_e}&\qbox{if $v=\iota(e)$}\\
0&\qbox{otherwise}\end{array}\right.$$
$$\qbox{and the $T(G_e)$-coordinate of
$\Res_f(M_v)$ is}\left\{\begin{array}{ll}
f_e^*(M_v)\res{G_v}{G_e}&\qbox{if $v=\ovl{\tau(e)}$}\\
0&\qbox{otherwise.}\end{array}\right.$$

We take the product of the maps on the coordinates, which is
well defined since each edge has only two vertex maps associate to it. The maps for
the groups of stable automorphisms are defined analogously and $\partial$
is the map defined in Equation~\eqref{eq:partial}.
\end{thm}

\begin{proof}
First we check that $\partial$ is constant on the cosets of the image
of $\Res-\Res_f$. Given $\psi_v\in \wh{\Aut}_{G_v}(k)$ for $v\in VY$,
we extend it to $v\in VT$ by $\psi_v=t^{-1}_v\psi_{\bar v}t_v$.
Then
$$(\Res-\Res_f)(\psi_v)=\big(\psi_{\iota(e)}\res{G_{\iota(e)}}{G_e}\big)
\big(\psi_{\tau(e)}\res{G_{\tau(e)}}{G_e}\big)^{-1}.$$
The maps $\psi_v$ combine in the obvious way to yield a map of
complexes
$$C(\tilde k\res{G}{G_v},\varphi_e)\longrightarrow
C\big(\tilde k\res {G}{G_v},\varphi_e(\psi_{\iota(e)}\res{G_{\iota(e)}}{G_e})
(\psi_{\tau(e)}\res{G_{\tau(e)}}{G_e})^{-1}\big),$$
in the sense that the diagram commutes in the stable module category.
It is also a stable isomorphism on the modules, and so the third
objects in the triangles are isomorphic, as required.

The image of $\partial$ is contained in $\ker(\Res)$ by
Lemma~\ref{lem1}, so we obtain a function
$$\ovl{\partial}:\coker(\Res-\Res_f)\longrightarrow\ker(\Res).$$
We shall construct an inverse $s$ as follows. Given $M\in T(G)$ such that $M\res G{G_v} \simeq  k$ for $v\in VY$, set $M_v=M\res G{G_v}$ and choose
isomorphisms
$\theta_v:\tilde k \res{G}{G_v}\longrightarrow M_v$; extend these to all the
vertices $VT$ by setting $\theta_v=s_v\otimes \theta_{\bar v}$.
Put
$\varphi'_e=
\big(\theta_{\tau(e)}\res{G_{\tau(e)}}{G_e}\big)^{-1}(\theta_{\iota(e)}\res{G_{\iota(e)}}{G_e})$
for $e\in EY$ and define $s(M)=(\varphi'_e)_{e \in EY}$.
As element of $\coker(\Res-\Res_f)$, $s(M)$ does not depend on the
choice of the maps $\theta_v$.
This function $s$ is actually a group homomorphism, because given
$\theta_j\in\wh{\Hom}_{G_e}(\tilde k,M_j)$ for $j=1,2$, we can combine
them to obtain a homomorphism
$$\xymatrix{k\ar[r]&k\otimes
k\ar[r]^<<<<{\theta_1\otimes\theta_2}&M_1\otimes M_2}.$$
Moreover, $\sigma\ovl\partial$ is the identity by Corollary~\ref{cor1},
since we can choose $\theta_v=\alpha_v^{-1}$ (the maps $\theta_v$ in
the statement of Corollary~\ref{cor1} are now the identity maps).

Given $M\in\ker(\Res)$ and $\theta_v$ as above, notice that
$C(\ul{\tilde k},\ul{\id})=k$, since this is the simplicial chain
complex on the tree $T$ and so $M$ is the third module in the
triangle for $C(\ul{\tilde k},\ul{\id})\otimes M$.
The maps $\theta_v$ provide an isomorphism of complexes
$$C(\ul{\tilde k},\ul{\id})\otimes M\longrightarrow
C\big(\ul{\tilde k},
\ul{\big(\theta_{\tau(e)}\res{G_{\tau(e)}}{G_e}\big)^{-1}(\theta_{\iota(e)}\res{G_{\iota(e)}}{G_e})}\big),$$ so $M = \ovl{\partial}\sigma(M)$ as required.
It follows that $\partial$ is a group homomorphism and that the
complex in the statement of the theorem is exact at
$\prod_{e\in EY}\wh{\Aut_{G_e}}(k)$ and  $T(G)$.

To check the exactness at
$\prod_{v\in VY}T(G_v)$, observe that if
$(M_v)_{v\in VY}\in\ker(\Res-\Res_f)$, then for each edge $e\in EY$
we have
$M_{\iota(e)}\res{G_{\iota(e)}}{G_e}\simeq f_e^*M_{\ovl{\tau(e)}}$;
let $\varphi_e:M_{\iota(e)}\res{G_{\iota(e)}}{G_e}\longrightarrow
M_{\tau(e)}\res{G_{\tau(e)}}{G_e}$ be a stable isomorphism.
Then $(M_v)_{v\in VY}=\Res\big(\ovl C(\ul M,\ul{\varphi})\big)$ by
Lemma~\ref{lem1}.

Exactness at $\prod_{v\in VY}\wh{\Aut}_{G_v}(k)$ is proved by an
argument analogous to that used for cohomology \cite[VII, \S9]{brown},
and is left to the reader.

\end{proof}

As special cases of Theorem~\ref{thm:trees}, we obtain
descriptions of $T(G)$ for amalgamated free products and HNN extensions
of groups of type $\Phi$.

\begin{cor}\label{cor:amalg}
Let $G=A*_CB$ be an amalgamated free product with 
 $A$ and $B$ of type $\Phi_k$. Then $G$ is of type $\Phi_k$ and there is 
an  exact sequence of abelian groups
$$\wh{\Aut}_{G}k\xrightarrow{\Res}
\wh{\Aut}_{A}k\oplus\wh{\Aut}_{B}k\xrightarrow{\Res-\Res_f}
\wh{\Aut}_{C}(k)\xrightarrow{\partial}
T(G)\xrightarrow{\Res}
T(A)\oplus T(B)\xrightarrow{\Res-\Res_f}
T(C),$$
where the maps are those defined in Theorem~\ref{thm:trees}.

If $C$ is finite and $k$ is a field or a complete discrete valuation ring then $\partial =0$.
\end{cor}

\begin{proof} The main part is a special case of Theorem~\ref{thm:trees}. If $C$ is finite then $\wh{\End}_{C}(k)  \cong \hat{H}^0(C;k) \cong k/|C|k$, so $\saut_C(k) \cong ( k/|C|k)^{\times}$. Given the restrictions on $k$, these units lift to units in $k$, so they are certainly in the image of restriction from $A$, thus restriction is onto and $\partial =0$.
	\end{proof}

\smallskip
\begin{example}\label{ex:sl2} 
Let $G=\SL(2,\Z)\cong C_6*_{C_2}C_4$. Then for $k$ a field or a complete discrete valuation ring  the sequence
$$\xymatrix{0\ar[r]&T(G)\ar[r]&T(C_6)\oplus T(C_4)\ar[r]&T(C_2)\ar[r]&0}$$
is exact. This follows directly from Corollary~\ref{cor:amalg}, except for exactness at $T(C_2)$, where it follows from the fact that $T(C_2)$ is generated by $\Omega k$, which is in the image of restriction from $T(C_4)$.
\begin{enumerate}
\item Suppose that $k$ is a field of characteristic 2.  Then $T(C_2) \cong 0$, $T(C_4) \cong \mathbb Z/2$ and $T(C_6) \cong \Hom(C_6,k^\times)$. It follows that $T(G) \cong \mathbb Z/6$ if $k$ contains a cube root of unity and $T(G) \cong \mathbb Z/2$ otherwise.
\item Suppose that $k$ is a field of characteristic 3. Then $T(C_4) \cong 0 \cong T(C_2)$ and $T(C_6) \cong \mathbb Z /2 \oplus \mathbb Z /2$ (using \cite{MT} for $T(C_6)$). Thus $T(G) \cong \mathbb Z /2 \times \mathbb Z /2$.
\end{enumerate}
\end{example}

Here is an example of inflation, which shows that inflation can be neither surjective nor injective.

\begin{example}\label{ex:inflation} 
Let $G=X*_ZY\cong C_4*_{C_2}C_4$ where $X$ and $Y$ are generated by $x$ and $y$ respectively.
Let $N$ be the normal subgroup generated by $xyxy^{-1}$, let $\bar{G}=G/N \cong Q_8$, and denote the quotient map by $\pi$.
The subgroup $N$ has
no torsion, since any torsion subgroup of $G$ must be conjugate to a subgroup of either $X$ or $Y$. Thus we do have an inflation map $\Inf_G^{\bar{G}}$.

Let $k$ be a field of characteristic $2$. Then $T(C_4) \cong \mathbb Z/2$ and $T(C_2) \cong 0$, so by Corollary~\ref{cor:amalg} $T(G) \cong \mathbb Z/2 \oplus \mathbb Z/2$. We claim that the image of $\Inf_G^{\bar{G}}$ is the diagonal subgroup of $T(G)$. The diagonal subgroup is generated by $\Omega k$, which is the image under inflation of the module of the same name in $T({\bar{G}})$. The elements of $T(G)$ are detected by restriction to the cyclic subgroups $X$ and $Y$, and $\Res ^G_X \Inf^{\bar{G}}_G =(\pi \! \restriction_X)^*  \Res^{\bar{G}}_{\bar{X}}$ and similarly for $Y$. But $\Res^{\bar{G}}_{\bar{X}}$  and $\Res^{\bar{G}}_{\bar{Y}}$ take the same value, 0 or 1: this can be seen by an easy calculation, since $T(Q_8)$ is generated by $\Omega k$ together with an explicit module of dimension 3 if $k$ contains a cube root of unity \cite[6.3]{CT2}. Both $\pi \! \restriction_X$ and $\pi \! \restriction_Y$ are isomorphisms, so $\Res ^G_X \Inf^{\bar{G}}_G = \Res ^G_Y \Inf^{\bar{G}}_G$.

The group $T(Q_8)$ has order at least 4 (it depends on $k$), so the inflation map is not injective.
\end{example}

We now turn to the case of HNN extensions. 
An HNN extension $G=H*_f$, for $A \leq H$ and $f \! : A\hookrightarrow H$, is the fundamental group of a graph of groups with one vertex $H$ and one edge $A$, where $A$ is included into $H$ at the initial vertex and is mapped by $f$ at the
terminal vertex. In terms of generators and relations, 
$G=\langle H,t \mid tat^{-1}=f(a)\;\forall\;a\in A\rangle
$.

\begin{cor}\label{cor:hnn}
Let $G=H*_f$ be a HNN extension with $H$ of type $\Phi_k$.
Then $G$ is of type $\Phi_k$ and there is an exact sequence of abelian groups
$$\wh{\Aut}_{G}k\xrightarrow{\Res}
\wh{\Aut}_{H}k\xrightarrow{\Res-\Res_f}
\wh{\Aut}_{A}(k)\xrightarrow{\partial}
T(G)\xrightarrow{\Res}
T(H)\xrightarrow{\Res-\Res_f}
T(A),$$
where the maps are those defined in Theorem~\ref{thm:trees}.
If $H$ is finite and $k$ is a field or a complete discrete valuation ring then $\partial$ is injective.
\end{cor}

\begin{proof} The main part is a special case of Theorem~\ref{thm:trees}. If $H$ is finite then $\Res \! : \saut_G(k) \rightarrow \saut_H(k)$ is onto by the same argument as in the previous proof. Thus $\Res - \Res_f=0$ and $\partial$ is injective.
\end{proof}

Note that if $f$ is the identity map on $H$ then $H*_f \cong \mathbb Z \times H$ and $\Res - \Res_f=0=0$

In the next three examples, $k$ is a field of characteristic $p$ or a complete discrete valuation ring with residue field of characteristic $p$ and $A$ is a finite group of order divisible by $p$. 

\begin{example} Consider $G=\mathbb Z \times A$ as an HNN extension and apply Corollary~\ref{cor:hnn}. Since $f$ is the identity, we have $\Res - \Res_f=0$, so there is a short exact sequence $0 \xrightarrow{} (k/|A|k)^{\times} \xrightarrow{\partial}
	T(G)\xrightarrow{\Res}
	T(A)\xrightarrow{} 0$; by Lemma~\ref{lem:inf} it is split by inflation.  Thus $T(G)$ is generated by modules inflated from $A$ and the rank 1 lattices.
	
	This shows that $T(G)$ can be very big: we cannot bound its cardinality independently of $k$.
	
	We want to calculate $\saut_G(k)$ for use in the next example. This can be done by noting that $\send_G(k) \cong \hat{H}^0(G;k)$ (the cup product agrees with the composition product \cite[VI.6, X.ex.6]{brown}) and using the spectral sequence in Tate-Farrell cohomology $H^p(\mathbb Z; \hat{H}^q(A;k)) \Rightarrow \hat{H}^{p+q}(G;k)$ from \cite[X.4 ex.5]{brown}. The result is that $\send_G(k) \cong k[x]/(x^2)$ as a $k$-algebra if $p=0$ in $k$, and $\send_G(k) \cong k/|A|k$ otherwise. Thus $\saut_G(k) \cong k^{\times} \oplus k$ as a group if $p=0$ in $k$ and  $\saut_G(k)\cong (k/|A|k)^{\times}$ otherwise. We can write this in a uniform fashion as $\saut_G(k) \cong (k/|A|k)^{\times} \oplus \tor_p(k)$.
	\end{example}

\begin{example} Consider $F=\mathbb Z \times G = \mathbb Z \times \mathbb Z \times A$. A similar calculation to the previous one shows that $T(F) \cong \saut_G(k) \oplus T(G)$. Combining this with our previous calculations we obtain $T(F) \cong \Hom(\mathbb Z \times \mathbb Z, (k/|A|k)^{\times}) \oplus \tor_p(k) \oplus T(A)$.
	
	The rank 1 lattices still appear, as $\Hom(\mathbb Z \times \mathbb Z, (k/|A|k)^{\times})$ and part of $T(A)$,  but with identifications. The image of inflation from $A$ accounts for the $T(A)$ summand. The modules corresponding to the $\tor_p(k)$ summand are more mysterious. They do not appear to be stably isomorphic to a lattice of finite rank. They also only occur when $p=0$ in $k$, so they cannot lift to characteristic 0, in contrast to the case for finite groups in Proposition~\ref{pr:lift}.  
	\end{example}

\begin{example}
	Consider $E=\mathbb Z \times (A*A)$. From Corollary~\ref{cor:amalg} we know that $T(A*A) \cong T(A) \oplus T(A)$. By Example~\ref{ex:cp*cp} we have $\saut_{A*A}(k) \cong \saut_A(k) \oplus \saut_A(k) \cong (k/|A|k)^{\times} \oplus (k/|A|k)^{\times}$.
	
	We obtain $T(G) \cong (k/|A|k)^{\times} \oplus (k/|A|k)^{\times} \oplus T(A) \oplus T(A)$. Only one copy of $(k/|A|k)^{\times}$ can be explained by the rank 1 lattices; in some sense this is because we should think instead of homomorphisms from $\mathbb Z$ to 
$\saut_{A*A}(k)$,
\end{example}

\end{document}